\documentclass{amsart}

\usepackage{graphicx}
\usepackage{float}

\usepackage[colorlinks]{hyperref}
\usepackage{cite}
\usepackage{cleveref}
\usepackage{enumitem}

\usepackage{amsmath,amssymb}

\usepackage[left=1.21in, right=1.21in]{geometry}

\newtheorem{theorem}{Theorem}[section]

\newtheorem{proposition}[theorem]{Proposition}
\newtheorem{corollary}[theorem]{Corollary}

\theoremstyle{definition}
\newtheorem{definition}[theorem]{Definition}
\newtheorem{example}[theorem]{Example}

\theoremstyle{remark}
\newtheorem{remark}[theorem]{Remark}

\numberwithin{equation}{section}

\begin{document}

\title{Soft Cone Metric Spaces and Some Fixed Point Theorems}

\author{\.{I}smet Alt\i nta\c{s}}
\address{Department of Mathematics, Faculty of Arts and Sciences, Sakarya University, Sakarya, 54187, Turkey}
%\curraddr{Department of Mathematics and Statistics,
%Case Western Reserve University, Cleveland, Ohio 43403}
\email{ialtintas@sakarya.edu.tr}

\author{Kemal Ta\c{s}k\"{o}pr\"{u}}
\address{Department of Mathematics, Faculty of Arts and Sciences, Bilecik \c{S}eyh Edebali University, Bilecik, 11000, Turkey}
\email{kemal.taskopru@bilecik.edu.tr}

%    General info

%\date{January 1, 2001 and, in revised form, June 22, 2001.}



\begin{abstract}
This paper is an introduction to soft cone metric spaces. We define the concept of soft cone metric via soft element, investigate soft converges in soft cone metric spaces and prove some fixed point theorems for contractive mappings on soft cone metric spaces.
\end{abstract}

\maketitle

\section{Introduction}\label{Sec1}
In 1999, Molodtsov \cite{Mol99} introduced the concept of soft sets as a new mathematical tool for dealing with uncertainties. He has shown several applications of this theory in solving many practical problems in economics, engineering, social science, medical science, etc. Then, Maji et al. \cite{MBR03,MR02} studied soft et theory in detail and applied it to decision making problems. In the line of reduction and addition of parameters of soft sets, some works have been done by Chen et al. \cite{CTYW05}, Pei and Miao \cite{PM05}, Kong et al. \cite{KJZW15}, Zou and Xiao \cite{ZX08}. Akta\c{s} and \c{C}a\u{g}man \cite{AC07} introduced the notion of soft group and discussed various properties of it. Shabir and Naz \cite{SN11} introduced the notion of soft topological space. Ayg\"{u}no\u{g}lu and Ayg\"{u}n \cite{AyA12} studied soft topological spaces and considered the the concept of soft point. Also, \c{C}etkin and Ayg\"{u}n \cite{CA16} studied convergence of soft nets in soft topological spaces. Das and Samanta introduced a notion of soft reel set and number \cite{DS12}, soft complex set and number \cite{DS13}, soft metric space \cite{DS13a,DS13b}, soft normed linear space \cite{DS13c,DS13d}. Chiney and Samanta \cite{CS15} introduced the concept of vector soft topology, Das, et al. studied on soft linear space and soft normed space \cite{DMS15}. G\"{u}ler et al. \cite{GYO16} introduced soft $G$-metric space.

On the other hand, in 2007, Huang and Zhang \cite{HZ07} introduced cone metric spaces with normal cone as a generalization of metric space. Rezapour and Hamlbarani \cite{RH08} presented the results of \cite{HZ07} for the case of cone metric space without normality in cone. In recent years, many authors workout on fixed point theorems in cone metric spaces and other generalizations of metric spaces (see \cite{MAK15, AAV13, AJ08, AlA12, ArAV12, AyAV12, AAB09, ABA10, AAB10, AM13, AMAR13, BAA09, CB11, GD03, HRS13 ,JKR11}).

In this paper, we introduce a concept of soft cone metric space which is based on soft elements. The plan of this paper is as follows: In \Cref{Sec2}, some preliminary definitions and results are given. In \Cref{Sec3}, we define the concept of soft cone metric according to soft element and described some of its properties. Also, we investigate soft converges in soft cone metric spaces. In \Cref{Sec4}, we prove some fixed point theorems for contractive mappings on soft cone metric spaces. 

\section{Preliminaries}\label{Sec2}

\begin{definition}
\cite{Mol99} Let $A$ be a set of parameters and $E$ be an initial universe. Let $\mathcal{P}(E)$ denote the power set of $E$. A pair $\left( F,A \right)$ is called a soft set over $E$, where $F$ is a mapping given by $F:A \rightarrow \mathcal{P}(E)$. In other words, a soft set over $E$ is a parametrized family of subsets of the universe $E$. For $\lambda \in A $, $F(\lambda)$ may be considered as the set of $\lambda$-approximate elements of the soft set $\left( F,A \right)$.
\end{definition}

\begin{definition}
\cite{MBR03} Let $\left( F,A \right)$ and $\left( G,A \right)$ be two soft sets over a common universe $E$.

\begin{enumerate}[label=(\alph*)]
\item $\left( F,A \right)$ is said to be null soft set, denoted by $\Phi$, if for all $\lambda \in A$, $F(\lambda)=\varnothing$. $\left( F,A \right)$ is said to an absolute soft set denoted by $\tilde{E}$, if for all $\lambda \in A$, $F(\lambda)=E$.

\item $\left( F,A \right)$ is said to be a soft subset of $\left( G,A \right)$ if for all $\lambda \in A$, $F(\lambda) \subseteq G(\lambda)$ and it is denoted by $\left( F,A \right) \tilde{\subseteq} \left( G,A \right)$. $\left( F,A \right)$ is said to be a soft upperset of $\left( G,A \right)$ if $\left( G,A \right)$ is a soft subset of $\left( F,A \right)$. We denote it by $\left( F,A \right) \tilde{\supseteq} \left( G,A \right)$. $\left( F,A \right)$ and $\left( G,A \right)$ is said to be equal if $\left( F,A \right)$ is a soft subset of $\left( G,A \right)$ and $\left( G,A \right)$ is a soft subset of $\left( F,A \right)$.

\item The intersection $\left( H,A \right)$ of $\left( F,A \right)$ and $\left( G,A \right)$ over $E$ is defined as $H(\lambda)=F(\lambda) \cap G(\lambda)$ for all $\lambda \in A$. We write $\left( F,A \right) \tilde{\cap} \left( G,A \right)=\left( H,A \right)$.

\item The union $\left( H,A \right)$ of $\left( F,A \right)$ and $\left( G,A \right)$ over $E$ is defined as $H(\lambda)=F(\lambda) \cup G(\lambda)$ for all $\lambda \in A$. We write $\left( F,A \right) \tilde{\cup} \left( G,A \right)=\left( H,A \right)$.

\item The cartesian product $(H,A)$ of $(F,A)$ and $(G,A)$ over $E$ denoted by $(H,A)=(F,A)\tilde{\times}(G,A)$, is defined as $H(\lambda)=F(\lambda) \times G(\lambda)$ for all $\lambda \in A$.

\item The difference $\left( H,A \right)$ of $\left( F,A \right)$ and $\left( G,A \right)$ over $E$ denoted by $\left( F,A \right) \tilde{\backslash} \left( G,A \right)=\left( H,A \right)$, is defined as $H(\lambda)=F(\lambda) \backslash G(\lambda)$ for all $\lambda \in A$.

\item The complement of $\left( F,A \right)$ is defined as $\left( F,A \right)^{c}=\left( F^{c},A \right)$, where $F^{c}:A \rightarrow \mathcal{P}(E)$ is a mapping given by  $F^{c}(\lambda)= B \backslash F(\lambda)$ for all $\lambda \in A$. Clearly, we have $\tilde{E}^{c}=\Phi$ and $\Phi^{c}=\tilde{E}$
\end{enumerate}

\end{definition}

\begin{definition}
\cite{DS12,DS13a} Let $A$ be a non-empty parameter set and $E$ be a non-empty set. Then a function $\epsilon:A \rightarrow E$ is said to be a soft element of $E$. A soft element $\epsilon$ of $E$ is said to belongs to a soft set $\left( F,A \right)$ of $E$ which is denoted by $\epsilon \tilde{\in}\left( F,A \right)$ if $\epsilon(\lambda) \in F(\lambda)$, $\forall \lambda \in A$. Thus for a soft set $\left( F,A \right)$ of $E$ with respect to the index set $A$, we have $F(\lambda)=\left\lbrace \epsilon(\lambda):\epsilon \tilde{\in} (F,A) \right\rbrace$, $\lambda \in A$. In that case, $\epsilon$ is also said to be a soft element of the soft set $\left( F,A \right)$. Thus every singleton soft set (a soft set $\left( F,A \right)$ of $E$ for which $F(\lambda)$ is a sigleton set, $\forall \lambda \in A$) can be identified with a soft element by simply identifying the singleton set with the element that it contains $\forall \lambda \in A$. 
\end{definition}

\begin{definition}
\cite{DS12,DS13a} Let $A$ be a set of parameters and $\mathbb{R}$ be the set of real numbers and $B(\mathbb{R})$ be the collection of all non-empty bounded subsets of $\mathbb{R}$. Then a mapping $F:A \rightarrow B(\mathbb{R})$ is called a soft real set, denoted by $\left( F,A \right)$. If specifically $\left( F,A \right)$ is a singleton soft set, then after identifying $\left( F,A \right)$ with the corresponding soft element, it will be called a soft real number.

The set of all soft real numbers is denoted by $\mathbb{R}(A)$ and the set of non-negative soft real numbers by $\mathbb{R}(A)^{*}$.

Throughout this paper, we consider the null soft set $\Phi$ and those soft sets $\left( F,A \right)$ over $E$ for which $F(\lambda) \neq \varnothing$, $\forall \lambda \in A$. We denote this collection by $S(\tilde{E})$. Thus for $\left( F,A \right)(\neq \varnothing) \in S(\tilde{E})$, $F(\lambda) \neq \varnothing$ for all $\lambda \in A$. For any soft set $\left( F,A \right) \in S(\tilde{E})$, the collection of all soft elements of $\left( F,A \right)$ is denoted by $SE \left( F,A \right)$.

Also, we use the notation $\tilde{x}$, $\tilde{y}$, $\tilde{z}$ to denote soft elements of soft set and $\tilde{r}$, $\tilde{s}$, $\tilde{t}$ to denote soft real numbers whereas $\bar{x}$, $\bar{y}$, $\bar{z}$ will denote a particular type of soft real numbers such that $\bar{r}(\lambda)=r$, for all $\lambda \in A$ etc. For $\tilde{r},\tilde{s} \in \mathbb{R}(A)$, 
\begin{align*}
\tilde{r}\tilde{\leq}\tilde{s}, \,\, \text{if} \,\, \tilde{r}(\lambda)\leq \tilde{s}(\lambda), \forall \lambda \in A, \quad  \quad  \tilde{r}\tilde{\geq}\tilde{s}, \,\, \text{if} \,\, \tilde{r}(\lambda)\geq \tilde{s}(\lambda), \forall \lambda \in A, \\
\tilde{r}\tilde{<}\tilde{s}, \,\, \text{if} \,\, \tilde{r}(\lambda)< \tilde{s}(\lambda), \forall \lambda \in A, \quad  \quad  \tilde{r}\tilde{>}\tilde{s}, \,\, \text{if} \,\, \tilde{r}(\lambda)> \tilde{s}(\lambda), \forall \lambda \in A.
\end{align*}
\end{definition}

\begin{proposition}
\cite{DS13a} \leavevmode
\begin{enumerate}[label=(\alph*)]
\item For any soft sets $\left( F,A \right),\left( G,A \right) \in S(\tilde{E})$, we have  $\left( F,A \right) \tilde{\subset} \left( G,A \right)$ if and only if every soft element of $\left( F,A \right)$ is also a soft elements of $\left( G,A \right)$.

\item Any collection of soft elements of a soft set can generate a soft subset of that soft set. The soft set constructed from a collection $\mathcal{B}$ of soft elements is denoted by $SS(\mathcal{B})$.

\item For any soft set $\left( F,A \right) \in S(\tilde{E})$, $SS(SE\left( F,A \right))=\left( F,A \right)$; whereas for a collection $\mathcal{B}$ of soft elements, $SE(SS(\mathcal{B})) \supset \mathcal{B}$, but, in general, $SE(SS(\mathcal{B})) \neq \mathcal{B}$.
\end{enumerate}
\end{proposition}

\begin{definition}
\cite{DMS15} Let $E$ be a vector space over a field $\mathbb{K}$ and $A$ be a parameter set. Let $\left( F,A \right)$ be a soft set over $E$. Now $\left( F,A \right)$ is said to be a soft vector space or soft linear space of $E$ over $\mathbb{K}$ if $F(\lambda)$ is a vector subspace of $E$, $\forall \lambda \in A$. The soft element of $\left( F,A \right)$ is said to be a soft vector of $\left( F,A \right)$. In a similar manner a soft element of the soft set $\left( \mathbb{K},A \right)$ is said to be a soft scalar, $\mathbb{K}$ being the scalar field.

Let $(F_{1},A),(F_{2},A),\dots,(F_{n},A)$ be $n$ soft sets over $E$. Then $(F,A)=(F_{1},A)+(F_{2},A)+\cdots+(F_{n},A)$ is a soft set over $E$ and is defined as $F(\lambda)=\lbrace x_{1}+x_{2}+\cdots+x_{n}:x_{i} \in F_{i}(\lambda),i=1,2,\dots,n \rbrace$, $\forall \lambda \in A$. Let $\alpha \in \mathbb{K}$ be a scalar and $(F,A)$ be a soft set over $E$, then $ (\alpha F,A)$ is a soft set over $E$ and is defined as $\alpha F(\lambda)=\lbrace \alpha x:x \in F(\lambda)\rbrace$, $\forall \lambda \in A$.

A soft vector $\tilde{x}$ in a soft vector space over $E$ is said to be the null soft vector if $\tilde{x}(\lambda)=\theta$, $\forall \lambda \in A$, $\theta$ being the zero element of $E$. It will be denoted by $\Theta$. Let $\tilde{x},\tilde{y}$ be soft vectors and $\tilde{\alpha}$ be a soft scalar. Then the addition $\tilde{x}+\tilde{y}$ of $\tilde{x},\tilde{y}$ and scalar multiplication $\tilde{\alpha}.\tilde{x}$ of $\tilde{\alpha}$ and $\tilde{x}$ are defined by $(\tilde{x}+\tilde{y})(\lambda)=\tilde{x}(\lambda)+\tilde{y}(\lambda)$ and $(\tilde{\alpha}.\tilde{x})(\lambda)=\tilde{\alpha}(\lambda).\tilde{x}(\lambda)$, $\forall \lambda \in A$. Obviously, $\tilde{x}+\tilde{y}$ and $\tilde{\alpha}.\tilde{x}$ are soft vectors of $\left( F,A \right)$. Also, $\bar{0}.\tilde{\alpha}=\Theta$ and $-\bar{1}.\tilde{\alpha}=-\tilde{\alpha}$, $\forall \tilde{\alpha}\tilde{\in}\tilde{E}$ and $\tilde{k}.\Theta=\Theta$, $\forall \tilde{k}\tilde{\in}\tilde{\mathbb{K}}$. However, $\tilde{k}\tilde{\alpha}=\Theta$ does not necessarily imply that either $\tilde{k}=\bar{0}$ or $\tilde{\alpha}=\Theta$.
\end{definition}

\begin{definition}
\cite{DMS15,DS13d} Let $\tilde{E}$ be the absolute soft vector space. Then a mapping $\Vert \cdot \Vert: SE(\tilde{E}) \rightarrow \mathbb{R}(A)^{*}$ is said to be a soft norm on the soft vector space $\tilde{E}$ if $\Vert \cdot \Vert$ satisfies the following conditions:
\begin{itemize}
\item[(N1)] $\Vert \tilde{x} \Vert \tilde{\geqslant} \bar{0} $, for all $\tilde{x} \tilde{\in} \tilde{E}$

\item[(N2)] $\Vert \tilde{x} \Vert = \bar{0} $  if and only if $\tilde{x}=\Theta$

\item[(N3)] $\Vert \tilde{\alpha}\tilde{x} \Vert = \vert \tilde{\alpha} \vert \Vert \tilde{x} \Vert$ for all $\tilde{x} \tilde{\in} \tilde{E}$ and for every soft scalar $\tilde{\alpha}$

\item[(N4)] $\Vert \tilde{x}+\tilde{y} \Vert \tilde{\leqslant} \Vert \tilde{x} \Vert + \Vert \tilde{y} \Vert$ for all $\tilde{x},\tilde{y} \tilde{\in} \tilde{E}$
\end{itemize}

The soft vector space $\tilde{E}$ with soft norm $\Vert \cdot \Vert$ on $\tilde{E}$ is said to be a soft normed linear space and is denoted by $(\tilde{E},\Vert \cdot \Vert,A)$ or $(\tilde{E},\Vert \cdot \Vert)$. (N1)-(N4) are said to be soft norm axioms.
\end{definition}

\begin{example}
\cite{DMS15} Let $\mathbb{R}(A)$ be the set of all soft real numbers. Define $\Vert \cdot \Vert : \mathbb{R}(A) \rightarrow \mathbb{R}(A)^{*}$ by $\Vert \tilde{x} \Vert=\vert \tilde{x} \vert$ for all $\tilde{x} \tilde{\in} \mathbb{R}(A)$, where $\vert \tilde{x} \vert$ denotes the modules of soft real number. Then, $\Vert \cdot \Vert$ is a soft norm on $\mathbb{R}(A)$ and $(\mathbb{R}(A),\Vert \cdot \Vert, A)$ is a soft normed linear space.
\end{example}

\begin{definition}
\cite{DS13d,DMS15} Let $(\tilde{E},\Vert \cdot \Vert,A)$ be a soft normed linear space and $\tilde{\epsilon} \tilde{>} \bar{0}$ be a soft real number.
\begin{enumerate}[label=(\alph*)]
\item We define the followings;
\begin{align*}
B(\tilde{x},\tilde{\epsilon})&=\left\lbrace \tilde{y} \tilde{\in} \tilde{E}:\Vert \tilde{x}-\tilde{y} \Vert \tilde{<} \tilde{\epsilon} \right\rbrace \subset SE(\tilde{E}), \\
\bar{B}(\tilde{x},\tilde{\epsilon})&=\left\lbrace \tilde{y} \tilde{\in} \tilde{E}:\Vert \tilde{x}-\tilde{y} \Vert \tilde{\leq} \tilde{\epsilon} \right\rbrace \subset SE(\tilde{E}), \\
S(\tilde{x},\tilde{\epsilon})&=\left\lbrace \tilde{y} \tilde{\in} \tilde{E}:\Vert \tilde{x}-\tilde{y} \Vert = \tilde{\epsilon} \right\rbrace \subset SE(\tilde{E}).
\end{align*}
$B(\tilde{x},\tilde{\epsilon})$, $\bar{B}(\tilde{x},\tilde{\epsilon})$ and $S(\tilde{x},\tilde{\epsilon})$ are called a open ball, a closed ball and a sphere with centre at $\tilde{x}$ and radius $\tilde{\epsilon}$, respectively. $SS(B(\tilde{x},\tilde{\epsilon}))$, $SS(\bar{B}(\tilde{x},\tilde{\epsilon}))$ and $SS(S(\tilde{x},\tilde{\epsilon}))$ are called a soft open ball, a soft closed ball and a soft sphere with centre at $\tilde{x}$ and radius $\tilde{\epsilon}$, respectively.

\item Let $(F,A)$ be a soft subset in $(\tilde{E},\Vert \cdot \Vert,A)$. Then $\tilde{x}$ is said to be an interior element of $(F,A)$ if $\exists$ a positive soft real number $\tilde{\epsilon}$ such that $\tilde{x} \in B(\tilde{x},\tilde{\epsilon}) \subset SE((F,A))$. All interior elements of soft set $(F,A)$ is denoted by $Int(F,A)$ and  $SS(Int(F,A))$ is said to be the soft interior of $(F,A)$.

\item Let $\mathcal{B}$ be a non-empty collection of soft elements of $\tilde{E}$ and $(F,A)$ be a soft subset in $(\tilde{E},\Vert \cdot \Vert,A)$. Then $\mathcal{B}$ is said to be open in $(\tilde{E},\Vert \cdot \Vert,A)$ if all elements of $\mathcal{B}$ are interior elements of $\mathcal{B}$. $(F,A)$ is said to be soft open in $(\tilde{E},\Vert \cdot \Vert,A)$ if there is a collection $\mathcal{B}$ of soft elements of $(F,A)$ such that $\mathcal{B}$ is open in $(\tilde{E},\Vert \cdot \Vert,A)$ and $(F,A)=SS(\mathcal{B})$.

\item A soft set $(F,A) \in S(\tilde{E})$, is said to be soft closed in $(\tilde{E},\Vert \cdot \Vert,A)$ if its complement
$(F,A)^{c}$ is a member of $S(\tilde{E})$ and is soft open in $(\tilde{E},\Vert \cdot \Vert,A)$.

\end{enumerate}
\end{definition}

\begin{definition}
\cite{DMS15,DS13d} \leavevmode
\begin{enumerate}[label=(\alph*)]
\item A sequence $\left\lbrace \tilde{x}_{n} \right\rbrace$ of soft elements in a soft normed linear space $(\tilde{E},\Vert \cdot \Vert,A)$  is said to be convergent and converges to a soft element $\tilde{x}$ if $\Vert \tilde{x}_{n}-\tilde{x} \Vert \to \bar{0}$ as $n \to \infty$. This means for every $\tilde{\epsilon} \tilde{>} \bar{0}$, chosen arbitrarily, $\exists$ a natural number $N=N(\tilde{\epsilon})$ such that $\bar{0} \tilde{\leqslant} \Vert \tilde{x}_{n}- \tilde{x} \Vert \tilde{<} \tilde{\epsilon}$, whenever $n>N$ i.e. $n>N \Rightarrow \tilde{x}_{n} \in B(\tilde{x},\tilde{\epsilon})$ ($B(\tilde{x},\tilde{\epsilon})$ is a open ball with centre $\tilde{x}$ and radius $\tilde{\epsilon}$).

\item A sequence $\left\lbrace \tilde{x}_{n} \right\rbrace$ of soft elements in a soft normed linear space $(\tilde{E},\Vert \cdot \Vert,A)$ is said to be a Cauchy sequence in $\tilde{E}$ if corresponding to every $\tilde{\epsilon} \tilde{>} \bar{0}$, $\exists$ a natural number $N=N(\tilde{\epsilon})$ such that $\Vert \tilde{x}_{n}-\tilde{x}_{m} \Vert \tilde{\leqslant} \tilde{\epsilon}$, $\forall m,n > N$ i.e. $\Vert \tilde{x}_{n}-\tilde{x}_{m} \Vert \to \bar{0}$ as $n,m \to \infty$.

\item Let $(\tilde{E},\Vert \cdot \Vert,A)$ be a soft normed linear space. Then $\tilde{E}$ is said to be complete if every Cauchy sequence of soft elements in $\tilde{E}$ converges to a soft element of $\tilde{E}$. Every complete soft normed linear space is called a soft Banach space.
\end{enumerate}
\end{definition}

\begin{theorem}\label{Teo2.11}
\cite{DMS15,DS13d} Every Cauchy sequence in $\mathbb{R}(A)$, where $A$ is a finite set of parameters, is convergent i.e. the set of all soft real numbers with its usual modulus soft norm with to finite set of parameters, is a soft Banach space.
\end{theorem}

\begin{definition}
\cite{DS13a} Let $X$ be a non-empty set and $A$ be non-empty a parameter set. A mapping $d: SE(\tilde{X}) \times SE(\tilde{X}) \rightarrow \mathbb{R}(A)^{*}$ is said to be a soft metric on the soft set $\tilde{X}$ if $d$ satisfies the following conditions:
\begin{enumerate}
\item[(M1)] $d(\tilde{x},\tilde{y}) \tilde{\geq} \bar{0}$, for all $\tilde{x},\tilde{y} \tilde{\in} \tilde{X}$.

\item[(M2)] $d(\tilde{x},\tilde{y}) = \bar{0}$ if and only if $\tilde{x}=\tilde{y}$.

\item[(M3)] $d(\tilde{x},\tilde{y})=d(\tilde{y},\tilde{x})$ for all $\tilde{x},\tilde{y} \tilde{\in} \tilde{X}$.

\item[(M4)] $d(\tilde{x},\tilde{y}) \tilde{\leq} d(\tilde{x},\tilde{z}) + d(\tilde{z},\tilde{y})$ for all $\tilde{x},\tilde{y}, \tilde{z} \tilde{\in} \tilde{X}$.
\end{enumerate}

The soft $\tilde{X}$ with a soft metric $d$ on $\tilde{X}$ is said to be a soft metric space and denoted by $(\tilde{X},d,A)$ or $(\tilde{X},d)$.
\end{definition}

\begin{proposition}
\cite{DMS15} Let $(\tilde{X},\Vert \cdot \Vert,A)$ be soft normed linear space. Let us define $d:\tilde{X}\times\tilde{X} \rightarrow \mathbb{R}(A)^{*}$ by $d(\tilde{x},\tilde{y})=\Vert \tilde{x}-\tilde{y} \Vert$ for all $\tilde{x},\tilde{y}\tilde{\in}\tilde{X}$. Then $d$ is a soft metric on $\tilde{X}$.
\end{proposition}

\begin{definition}
\cite{HZ07} \leavevmode
\begin{enumerate}[label=(\alph*)]
\item Let $E$ be a real Banach space and $P$ be a subset of $E$. $P$ is called a cone if and only if 
\begin{itemize}
\item[(1)] $P$ is closed, non-empty and $P \neq \left\lbrace \theta \right\rbrace$,

\item[(2)] $a,b \in \mathbb{R}$, $a,b \geqslant 0$, $x,y \in P$ $\Rightarrow$ $ax+by \in P$,

\item[(3)] $x \in P$ and $-x \in P $ $\Rightarrow$ $x=\theta$.
\end{itemize}

\item For a given cone $P \subseteq E$, we can define a partial ordering $\preceq$ with respect to $P$ by $x \preceq y$ if and only if $y-x \in P$. $x \prec y$ will stand for $x \preceq y$ and $x \neq y$, while $x \ll y$ will stand for $y-x \in intP$, where $intP$ denotes the interior of P.

\item The cone $P$ is called normal if there is a number $\alpha>0$ such that for all $x,y \in E$, we have $\theta \preceq x \preceq y$ implies $\Vert x \Vert \leq \alpha \Vert y \Vert$. The least positive number satisfying this inequality is called the normal constant of $P$. The cone $P$ is called regular if every increasing sequence which is bounded from above is convergent. Equivalently the cone $P$ is called regular if every decreasing sequence which is bounded from below is convergent. Regular cones are normal and there exist normal cones which are not regular.

\item Let $X$ be a non-empty set. Suppose the mapping $d: X \times X \rightarrow E$ satisfies
\begin{enumerate}
\item[(d1)] $\theta \prec d(x,y)$ for all $x,y \in X$ and $d(x,y)=\theta$ if and only if $x=y$.

\item[(d2)] $d(x,y)=d(y,x)$ for all $x,y \in X$.

\item[(d3)] $d(x,y) \preceq d(x,z)+d(y,z)$ for all $x,y,z \in X$.
\end{enumerate}
Then $d$ is called a cone metric on $X$ and $(X,d)$ is called a cone metric space.

\end{enumerate}

\end{definition}

\section{Soft Cone Metric Spaces}\label{Sec3}

\begin{definition}
Let $(\tilde{E},\Vert \cdot \Vert,A)$ be a soft real Banach space and $(P,A) \in S(\tilde{E})$ be a soft subset of $\tilde{E}$. Then $(P,A)$ is called a soft cone if and only if
\begin{enumerate}
\item $(P,A)$ is closed, $(P,A) \neq \Phi$ and $(P,A) \neq SS(\lbrace \Theta \rbrace)$,

\item $\tilde{a},\tilde{b} \in \mathbb{R}(A)^{*}$, $\tilde{x},\tilde{y} \tilde{\in} (P,A)$ implies $\tilde{a}\tilde{x}+\tilde{b}\tilde{y} \tilde{\in} (P,A)$,

\item $\tilde{x} \tilde{\in} (P,A)$ and $-\tilde{x} \tilde{\in}(P,A)$ implies $\tilde{x}=\Theta$.
\end{enumerate}
\end{definition}

Given a soft cone $(P,A) \in S(\tilde{E})$, we define a soft partial ordering $\tilde{\preceq}$ with respect to $(P,A)$ by $\tilde{x} \tilde{\preceq} \tilde{y}$ if and only if $\tilde{y}-\tilde{x} \tilde{\in} (P,A)$. We write $\tilde{x} \tilde{\prec} \tilde{y}$ to indicate that $\tilde{x} \tilde{\preceq} \tilde{y}$ but $\tilde{x} \neq \tilde{y}$, while $\tilde{x} \tilde{\ll} \tilde{y}$ will stand for $\tilde{y}-\tilde{x} \tilde{\in} Int(P,A)$, $Int(P,A)$ denotes the interior of $(P,A)$.

\begin{definition}
The soft cone $(P,A)$ in soft real Banach space $\tilde{E}$ is called
\begin{enumerate}[label=(\alph*)]
\item normal, if there is a soft real number $\tilde{\alpha}\tilde{>}\bar{0}$ such that for all $\tilde{x},\tilde{y} \in \tilde{E}$, $\Theta \tilde{\preceq} \tilde{x} \tilde{\preceq} \tilde{y}$ implies $\Vert \tilde{x} \Vert \tilde{\leq} \tilde{\alpha} \Vert \tilde{y} \Vert$, where $\tilde{\alpha}$ is called soft constant of $(P,A)$.

\item minihedral, if $sup(\tilde{x},\tilde{y})$ exists for all $\tilde{x},\tilde{y} \tilde{\in} \tilde{E}$.

\item strongly minihedral, if every soft set in $\tilde{E}$ which is bounded from above has a supremum.

\item solid, if $Int(P,A) \neq \Phi$.

\item regular, if every increasing sequence of soft elements in $\tilde{E}$ which is bounded from above is convergent. That is, if $\lbrace \tilde{x}_{n} \rbrace$ is a sequence of soft elements in $\tilde{E}$ such that $\tilde{x}_{1} \tilde{\preceq} \tilde{x}_{2} \tilde{\preceq} \cdots \tilde{\preceq} \tilde{x}_{n} \cdots $ for some soft elements in $\tilde{E}$ then there is $\tilde{x} \tilde{\in} \tilde{E}$ such that $\Vert \tilde{x}_{n} - \tilde{x} \Vert \to \bar{0}$ as $n \to \infty$. Equivalently, the soft cone $(P,A)$ is regular if and only if every decreasing sequence of soft elements in $\tilde{E}$ which is bounded from below is convergent.
\end{enumerate}
\end{definition}

\begin{example}
Let $\mathbb{R}(A)$ be all soft real number, where $A$ is a finite set of parameters. Let $\mathbb{R}^{n}(A)=\mathbb{R}(A) \times\mathbb{R}(A) \times \cdots \times\mathbb{R}(A)$. Then, $\mathbb{R}^{n}(A)$ is a soft Banach space as a result of \Cref{Teo2.11}. Let $\tilde{E}=\mathbb{R}^{n}(A)$ with $(P,A)=SS \left\lbrace (\tilde{x}_{1},\tilde{x}_{2},\dots,\tilde{x}_{n}): \tilde{x}_{i} \tilde{\geq} \bar{0}, \forall i=1,2,\dots,n \right\rbrace$. Then the soft cone $(P,A)$ is normal, minihedral, strongly minihedral and solid.
\end{example}

In the following we always suppose $(P,A)$ is a soft cone in soft Banach space $\tilde{E}$ with $Int(P,A) \neq \Phi$ and $\tilde{\preceq}$ is soft partial ordering with respect to $(P,A)$.

\begin{definition}
Let $X$ be a non-empty set and $\tilde{X}$ be absolute soft set. A mapping $d:SE(\tilde{X}) \times SE(\tilde{X}) \rightarrow SE(\tilde{E})$ is said to be a soft cone metric on $\tilde{X}$ if $d$ satisfies the following axioms:

\begin{enumerate}
\item[(d1)]  $\Theta \tilde{\prec} d(\tilde{x},\tilde{y})$ for all $\tilde{x},\tilde{y} \tilde{\in} \tilde{X}$ and $d(\tilde{x},\tilde{y})=\Theta$ if and only if $\tilde{x}=\tilde{y}$.

\item[(d2)] $d(\tilde{x},\tilde{y})=d(\tilde{y},\tilde{x})$ for all $\tilde{x},\tilde{y} \tilde{\in} \tilde{X}$.

\item[(d3)] $d(\tilde{x},\tilde{y}) \tilde{\preceq} d(\tilde{x},\tilde{z})+d(\tilde{z},\tilde{y})$ for all $\tilde{x},\tilde{y},\tilde{z} \tilde{\in} \tilde{X}$.
\end{enumerate}
Then, the soft set $\tilde{X}$ with a soft cone metric $d$ on $\tilde{X}$ is called a soft cone metric space and is denoted by $(\tilde{X},d,A)$. 
\end{definition}

It is obvious that soft cone metric spaces generalize soft metric spaces.

\begin{example}
Let $A$ be a finite set of parameters, $\tilde{E}=\mathbb{R}^{2}(A)$, $(P,A)=SS \left\lbrace (\tilde{x},\tilde{y}) \tilde{\in} \tilde{E}: \tilde{x},\tilde{y} \tilde{\geq} \bar{0} \right\rbrace$, $\tilde{X}=\mathbb{R}(A)$ and $d:\tilde{X} \times \tilde{X} \rightarrow \tilde{E}$ such that $d(\tilde{x},\tilde{y})=\left( \vert \tilde{x}-\tilde{y} \vert,\tilde{\alpha}\vert \tilde{x}-\tilde{y} \vert  \right)$, where $\tilde{\alpha}\tilde{\geq}\bar{0}$ is soft constant. Then, $(\tilde{X},d,A)$ is a soft cone metric space.
\end{example}

\begin{example}\label{Ex}
Every parametrized family of crisp cone metrics $\lbrace d_{\lambda}:\lambda \in A \rbrace$ on a crisp set $X$ can be considered as a soft cone metric on the sof set $\tilde{X}$.
\begin{proof}
Let $\tilde{x},\tilde{y}\tilde{\in}\tilde{X}$, then $\tilde{x}(\lambda),\tilde{y}(\lambda) \in X$, $\forall \lambda \in A$. Let us define a mapping $d:SE(\tilde{X}) \times SE(\tilde{X}) \rightarrow SE(\tilde{E})$ by $d(\tilde{x},\tilde{y})(\lambda)=d_{\lambda}(\tilde{x}(\lambda),\tilde{y}(\lambda))$, $\forall \lambda \in A$, $\forall \tilde{x},\tilde{y}\tilde{\in}\tilde{X}$. Then, $d$ is a soft cone metric on $\tilde{X}$. We now verify the axioms (d1)-(d3) for soft cone metric.
\begin{itemize}
\item[(d1)] We have $\theta \prec d(\tilde{x},\tilde{y})(\lambda)=d_{\lambda}\left( \tilde{x}(\lambda),\tilde{y}(\lambda) \right)$, $\forall \lambda \in A$, $\forall \tilde{x},\tilde{y}\tilde{\in}\tilde{X}$. Thus $\Theta \prec d(\tilde{x},\tilde{y})$, where $\Theta(\lambda)=\theta$, $\forall \lambda \in A$. Also, 
\begin{align*}
d(\tilde{x},\tilde{y})(\lambda)&=d_{\lambda}\left( \tilde{x}(\lambda),\tilde{y}(\lambda) \right)=\theta, \forall \lambda \in A \\
&\Leftrightarrow \Vert d_{\lambda}\left( \tilde{x}(\lambda),\tilde{y}(\lambda) \right) \Vert =\Vert \theta \Vert \\
&\Leftrightarrow  \Vert d_{\lambda}\left( \tilde{x}(\lambda),\tilde{y}(\lambda) \right) \Vert =0 \\
&\Leftrightarrow \tilde{x}(\lambda)=\tilde{y}(\lambda),  \forall \lambda \in A \\
&\Leftrightarrow \tilde{x}=\tilde{y}.
\end{align*}

\item[(d2)] $d(\tilde{x},\tilde{y})(\lambda)=d_{\lambda}\left( \tilde{x}(\lambda),\tilde{y}(\lambda) \right)=d_{\lambda}\left( \tilde{y}(\lambda),\tilde{x}(\lambda) \right)=d(\tilde{x},\tilde{y})(\lambda)$, $\forall \lambda \in A \Rightarrow d(\tilde{x},\tilde{y})=d(\tilde{y},\tilde{x})$, $\forall \tilde{x},\tilde{y}\tilde{\in}\tilde{X}$.

\item[(d3)] For all $\tilde{x},\tilde{y},\tilde{z}\tilde{\in}\tilde{X}$,
\begin{align*}
\left[ d(\tilde{x},\tilde{z})+d(\tilde{z},\tilde{y}) \right] (\lambda)&=d\left( \tilde{x},\tilde{z} \right)(\lambda)+d\left( \tilde{z},\tilde{y} \right)(\lambda) \\
&= d_{\lambda}\left( \tilde{x}(\lambda),\tilde{z}(\lambda) \right)+d_{\lambda}\left( \tilde{z}(\lambda),\tilde{y}(\lambda) \right) \\
&\geq d_{\lambda}\left( \tilde{x}(\lambda),\tilde{y}(\lambda) \right) \\
&=d\left( \tilde{x},\tilde{y} \right)(\lambda), \forall \lambda \in A.
\end{align*}
So, $d\left( \tilde{x},\tilde{y} \right) \preceq d\left( \tilde{x},\tilde{z} \right)+d\left( \tilde{z},\tilde{y} \right)$.
\end{itemize}
Thus, $d$ is soft cone metric on $\tilde{X}$.
\end{proof}
\end{example}

\begin{example}
Every crisp cone metric $\rho$ on a crisp set $X$ can be extended to a soft cone metric on $\tilde{X}$.
\begin{proof}
First we construct the absolute soft set and the soft Banach space $\tilde{E}$ using the non-empty set of parameters $A$. Let us define  a mapping $d:SE(\tilde{X}) \times SE(\tilde{X}) \rightarrow SE(\tilde{E})$ by $d(\tilde{x},\tilde{y})(\lambda)=\rho(\tilde{x}(\lambda),\tilde{y}(\lambda))$, $\forall \lambda \in A$, $\forall \tilde{x},\tilde{y}\tilde{\in}\tilde{X}$.
Then, using the same procedure as \Cref{Ex}, it can be easily prove that $d$ is a soft cone metric on $\tilde{X}$.
\end{proof}
\end{example}

The soft cone metric defined using the crisp cone metric $\rho$ is called the soft cone metric generated by $\rho$.

\begin{remark}
The converse of the \Cref{Ex} is not true. Though every parametrized family of crisp cone metrics can be considered as a soft cone metric but any soft cone metric is not just a parametrized family of crisp cone metrics. Thus, the soft cone metric is more general and comprehensive than any parametrized family of the crisp cone metrics.
\end{remark}

\begin{theorem}
If a soft cone metric space $d$ on $\tilde{X}$ satisfies the following axiom
\begin{enumerate}
\item[(d4)] \label{(d4)} For $(r,s)\in X \times X$ and $\lambda \in A$, $\left\lbrace d(\tilde{x},\tilde{y})(\lambda):\tilde{x}(\lambda)=r,\tilde{y}(\lambda)=s \right\rbrace$ is a singleton set and if for $\lambda \in A$, $d_{\lambda}:X \times X \rightarrow E$ is defined by $d_{\lambda}\left( \tilde{x}(\lambda),\tilde{y}(\lambda) \right)=d(\tilde{x},\tilde{y})(\lambda)$, $\tilde{x},\tilde{y} \in \tilde{X}$. 
\end{enumerate}
then $d_{\lambda}$ is a cone metric on $X$.
\end{theorem}

\begin{proof}
Clearly $d_{\lambda}:X \times X \rightarrow E$ is rule that assigns an ordered pair of $X$ to a crisp element $e \in E$ with $\theta \preceq e$, $\forall \lambda \in A$. Now the well defined property of $d_{\lambda}$, $\forall \lambda \in A$, follows from the axiom (d4) and the soft cone metric axioms gives the cone metric conditions of $d_{\lambda}$, $\forall \lambda \in A$. Thus the soft cone metric satisfying (d4) gives a parametrized family of crisp cone metrics. With this point of view, it also follows that a soft cone metric, satisfying (d4) is a particular soft cone mapping as defined in \cite{MS10}, where $d:A \rightarrow (E^{*})^{X \times X}$ ($E^{*}$ denotes the set of all $e \in E$ with $\theta \preceq e$).
\end{proof}

\begin{definition}
Let $(\tilde{X},d,A)$ be soft cone metric space. Let $\lbrace \tilde{x}_{n} \rbrace$ be a sequence of soft elements in $\tilde{X}$ and $\tilde{x}\tilde{\in}\tilde{X}$. If for all $\tilde{c}\tilde{\in}\tilde{E}$ with $\Theta\tilde{\ll}\tilde{c}$ there is a natural number $N$ such that for all $ n > N$,  $d(\tilde{x}_{n},\tilde{x}) \tilde{\ll} \tilde{c}$, then $\lbrace \tilde{x}_{n} \rbrace$ is said to be convergent and $\lbrace \tilde{x}_{n} \rbrace$ converges to $\tilde{x}$ and $\tilde{x}$ is the limit of $\lbrace \tilde{x}_{n} \rbrace$. We denote this by $\displaystyle{\lim_{n \to \infty}\tilde{x}_{n}}=\tilde{x}$ or $\tilde{x}_{n} \to \tilde{x}$ as $n \to \infty$.
\end{definition}

\begin{theorem}
Let $(\tilde{X},d,A)$ a soft cone metric space and $(P,A)$ be a normal soft cone with normal soft constant $\tilde{\alpha}$. Let $\lbrace \tilde{x}_{n} \rbrace$ be a sequence of soft elements in $\tilde{X}$. Then, $\lbrace \tilde{x}_{n} \rbrace$ converges to $\tilde{x}$ if and only if $d(\tilde{x}_{n},\tilde{x}) \to \Theta$ as $n \to \infty$.
\end{theorem}

\begin{proof}
Suppose that $\lbrace \tilde{x}_{n} \rbrace$ converges $\tilde{x}$. For every soft real $\tilde{\epsilon}\tilde{>}\bar{0}$, choose $\tilde{c}\tilde{\in}\tilde{E}$ with $\Theta\tilde{\ll}\tilde{c}$ and $\tilde{\alpha}\Vert \tilde{c} \Vert \tilde{<} \tilde{\epsilon}$. Then, there is a natural number $N$ such that for all $n>N$, $d(\tilde{x}_{n},\tilde{x})\tilde{\ll}\tilde{c}$. So that when $n>N$, $\Vert d(\tilde{x}_{n},\tilde{x}) \Vert \tilde{\leq} \tilde{\alpha}\Vert \tilde{c} \Vert \tilde{<} \tilde{\epsilon}$. This means $d(\tilde{x}_{n},\tilde{x}) \to \Theta$ as $n \to \infty$. Conversely, suppose that $d(\tilde{x}_{n},\tilde{x}) \to \Theta$ as $n \to \infty$. For $\tilde{c}\tilde{\in}\tilde{E}$ with $\Theta\tilde{\ll}\tilde{c}$, there is $\tilde{\delta}\tilde{>}\bar{0}$ such that $\Vert \tilde{x} \Vert \tilde{<} \tilde{\delta}$ implies $\tilde{c}-\tilde{x}\tilde{\in} Int(P,A)$. For this $\tilde{\delta}$, there is a natural number $N$ such that for all $n>N$, $\Vert d(\tilde{x}_{n},\tilde{x}) \Vert \tilde{<}\tilde{\delta}$. This means $d(\tilde{x}_{n},\tilde{x})  \tilde{\ll}\tilde{c}$. Therefore, $\lbrace \tilde{x}_{n} \rbrace$ converges to $\tilde{x}$.
\end{proof}

\begin{theorem}
Let $(\tilde{X},d,A)$ a soft cone metric space and $(P,A)$ be a normal soft cone with normal soft constant $\tilde{\alpha}$. Let $\lbrace \tilde{x}_{n} \rbrace$ be a convergent sequence of soft elements in $\tilde{X}$. Then, the limit of $\lbrace \tilde{x}_{n} \rbrace$ is unique.
\end{theorem}

\begin{proof}
If possible let there exist a sequence $\lbrace \tilde{x}_{n} \rbrace$  of soft elements in $\tilde{X}$ such that $\tilde{x}_{n} \to \tilde{x}$ as $n \to \infty$ and $\tilde{x}_{n} \to \tilde{y}$ as $n \to \infty$, where $\tilde{x},\tilde{y}\tilde{\in}\tilde{X}$ with $\tilde{x} \neq \tilde{y}$. Then there is at least one $\lambda \in A$ such that $d(\tilde{x},\tilde{y})(\lambda) \neq \theta$. We consider a element $c_{\lambda}$ of crisp Banach space $E$ with $\theta \ll c_{\lambda}$ satisfying $0<\Vert c_{\lambda} \Vert < \frac{1}{2}\Vert
d(\tilde{x},\tilde{y})(\lambda) \Vert$. Let $\tilde{c}\tilde{\in}\tilde{E}$ with $\tilde{c}(\lambda)=c_{\lambda}$, $\forall \lambda \in A$. Since $\tilde{x}_{n} \to \tilde{x}$, $\tilde{x}_{n} \to \tilde{y}$, there is a natural number $N$ such that for all $n>N$, $d(\tilde{x}_{n},\tilde{x})\tilde{\ll}\tilde{c} \Rightarrow d(\tilde{x}_{n},\tilde{x})(\lambda) \ll c_{\lambda}$ and $d(\tilde{x}_{n},\tilde{y})\tilde{\ll}\tilde{c} \Rightarrow d(\tilde{x}_{n},\tilde{y})(\lambda) \ll c_{\lambda}$, in particular. We have 
\begin{equation*}
d(\tilde{x},\tilde{y})(\lambda) \preceq d(\tilde{x}_{n},\tilde{x})(\lambda)+d(\tilde{x}_{n},\tilde{y})(\lambda) \ll 2c_{\lambda}.
\end{equation*}
Hence $\Vert d(\tilde{x},\tilde{y})(\lambda) \Vert \leq 2\tilde{\alpha}(\lambda) \Vert c_{\lambda} \Vert$ or $\Vert c_{\lambda} \Vert > \frac{1}{2\tilde{\alpha}(\lambda)}\Vert d(\tilde{x},\tilde{y})(\lambda) \Vert$. This is contradicted by the fact that $c_{\lambda}$ is arbitrary and therefore $\tilde{c}$ is arbitrary. Thus the result follows.
\end{proof}

\begin{definition}
Let $(\tilde{X},d,A)$ a soft cone metric space and $\lbrace \tilde{x}_{n} \rbrace$ be a sequence of soft elements in $\tilde{X}$. If for any $\tilde{c}\tilde{\in}\tilde{E}$ with $\Theta \tilde{\ll}\tilde{c}$, there is a natural number $N$ such that for all $n,m>N$, $d(\tilde{x}_{n},\tilde{x}_{m})\tilde{\ll}\tilde{c}$, then $\lbrace \tilde{x}_{n} \rbrace$ is called a Cauchy sequence in $\tilde{X}$.
\end{definition}

\begin{definition}
Let $(\tilde{X},d,A)$ a soft cone metric space. If every Cauchy sequence of soft elements in $\tilde{X}$ is convergent in $\tilde{X}$, then $(\tilde{X},d,A)$ is called complete soft cone metric space.
\end{definition}

\begin{theorem}
Let $(\tilde{X},d,A)$ a soft cone metric space. Then every convergent sequence of soft elements in $\tilde{X}$ is a Cauchy sequence.
\end{theorem}

\begin{proof}
Let $\lbrace \tilde{x}_{n} \rbrace$ be a sequence of soft elements in $\tilde{X}$ and $\tilde{x}_{n} \to \tilde{x}$ as $n \to \infty$. Then, for any $\tilde{c}\tilde{\in}\tilde{E}$ with $\Theta \tilde{\ll} \tilde{c}$ there is a natural number $N$ such that for $n,m>N$, $d(\tilde{x}_{n},\tilde{x})\tilde{\ll}\tilde{\frac{c}{2}}$ and $d(\tilde{x}_{m},\tilde{x})\tilde{\ll}\tilde{\frac{c}{2}}$. Hence
\begin{equation*}
d(\tilde{x}_{n},\tilde{x}_{m}) \tilde{\preceq} d(\tilde{x}_{n},\tilde{x})+d(\tilde{x}_{m},\tilde{x}) \tilde{\ll} 2\tilde{c}.
\end{equation*}
Therefore $\lbrace \tilde{x}_{n} \rbrace$ is a Cauchy sequence.
\end{proof}

\begin{theorem}
Let $(\tilde{X},d,A)$ a soft cone metric space and $(P,A)$ be a normal soft cone with normal soft constant $\tilde{\alpha}$. Let $\lbrace \tilde{x}_{n} \rbrace$ be a sequence of soft elements in $\tilde{X}$. Then, $\lbrace \tilde{x}_{n} \rbrace$ is a Cauchy sequence if and only if $d(\tilde{x}_{n},\tilde{x}_{m}) \to \Theta$ as $n,n \to \infty$.
\end{theorem}

\begin{proof}
Suppose that $\lbrace \tilde{x}_{n} \rbrace$ is a Cauchy sequence. For every soft real number $\tilde{\epsilon}\tilde{>}\bar{0}$, choose  $\tilde{c}\tilde{\in}\tilde{E}$ with $\Theta \tilde{\ll} \tilde{c}$ and $\tilde{\alpha} \Vert \tilde{c} \Vert \tilde{<} \tilde{\epsilon}$. Then, there is a natural number $N$, such that for all $nim>N$, $d(\tilde{x}_{n},\tilde{x}_{m})\tilde{\ll}\tilde{c}$. So, $\Vert d(\tilde{x}_{n},\tilde{x}_{m}) \Vert \tilde{\leq} \tilde{\alpha} \Vert \tilde{c} \Vert \tilde{<} \tilde{\epsilon}$. This means $d(\tilde{x}_{n},\tilde{x}_{m}) \to \Theta$ as $n,m \to \infty$.

Conversely, suppose that $d(\tilde{x}_{n},\tilde{x}_{m}) \to \Theta$ as $n,m \to \infty$. For $\tilde{c}\tilde{\in}\tilde{E}$ with $\Theta \tilde{\ll}\tilde{c}$, there is $\tilde{\delta}\tilde{>}\bar{0}$ such that $\Vert \tilde{x} \Vert \tilde{<} \tilde{\delta}$ implies $\tilde{c}-\tilde{x} \tilde{\in} Int(P,A)$. For this $\tilde{\delta}$, there is $N$ such that for all $n,m>N$, $\Vert d(\tilde{x}_{n},\tilde{x}_{m}) \Vert \tilde{<} \tilde{\delta}$. So, $\tilde{c}-d(\tilde{x}_{n},\tilde{x}_{m}) \tilde{\in} Int(P,A)$. This means $d(\tilde{x}_{n},\tilde{x}_{m})\tilde{\ll}\tilde{c}$. Therefore $\lbrace \tilde{x}_{n} \rbrace$ is a Cauchy sequence.
\end{proof}

\section{Fixed Point Theorems on Soft Cone Metric Spaces}\label{Sec4}

In this section, we prove some fixed point theorems of contractive mappings on soft cone metric spaces.

\begin{definition}
Let $(\tilde{X},d,A)$ be a soft cone metric space and $T:(\tilde{X},d,A) \rightarrow (\tilde{X},d,A)$ be a mapping. If there exists a soft element $\tilde{x}_{0}\tilde{\in}\tilde{X}$ such that $T\tilde{x}_{0}=\tilde{x}_{0}$, then $\tilde{x}_{0}$ is called a fixed element of $T$.
\end{definition}

\begin{definition}
Let $(\tilde{X},d,A)$ be a soft cone metric space and $T:(\tilde{X},d,A) \rightarrow (\tilde{X},d,A)$ be a mapping. For every $\tilde{x}_{0}\tilde{\in}\tilde{X}$, we can construct the sequence $\lbrace \tilde{x}_{n} \rbrace$ of soft elements by choosing $\tilde{x}_{0}$ and continuing by:
\begin{equation*}
\tilde{x}_{1}=T\tilde{x}_{0}, \quad \tilde{x}_{2}=T\tilde{x}_{1}=T^{2}\tilde{x}_{0}, \quad \dots \quad \tilde{x}_{n}=T\tilde{x}_{n-1}=T^{n}\tilde{x}_{0}, \quad \dots
\end{equation*}
We say that the sequence $\lbrace \tilde{x}_{n} \rbrace$ is constructed by iteration method.
\end{definition}

\begin{definition}
Let $(\tilde{X},d,A)$ be a soft cone metric space and $T:(\tilde{X},d,A) \rightarrow (\tilde{X},d,A)$ be a mapping. If there is a positive soft real number $\tilde{t}$ with $\bar{0}\tilde{\leq}\tilde{t}\tilde{<}\bar{1}$ such that
\begin{equation*}
d\left( T\tilde{x},T\tilde{y} \right) \tilde{\preceq} \tilde{t} d\left( \tilde{x},\tilde{y} \right), \,
\forall \tilde{x},\tilde{y}\tilde{\in}\tilde{X},
\end{equation*}
then $T$ is called contractive mapping in $\tilde{X}$.
\end{definition}

\begin{theorem}\label{Teo4.4}
Let $(\tilde{X},d,A)$ be a complete soft cone metric space and $T:(\tilde{X},d,A) \rightarrow (\tilde{X},d,A)$ be a contractive mapping. Then $T$ has a unique fixed soft element in $\tilde{X}$. For each $\tilde{x}\tilde{\in}\tilde{X}$, the iterative sequence $\lbrace T^{n}\tilde{x} \rbrace$ converges to the fixed soft element.
\end{theorem}

\begin{proof}
Let $(\tilde{X},d,A)$ be a complete soft cone metric space and $T:(\tilde{X},d,A) \rightarrow (\tilde{X},d,A)$ be a contractive mapping. Then, there is a positive soft real number $\tilde{t}$ with $\bar{0}\tilde{\leq}\tilde{t}\tilde{<}\bar{1}$ such that
\begin{equation*}
d\left( T\tilde{x},T\tilde{y} \right) \tilde{\preceq} \tilde{t} d\left( \tilde{x},\tilde{y} \right), \,
 \forall \tilde{x},\tilde{y}\tilde{\in}\tilde{X}.
\end{equation*}
For each $\tilde{x}_{0}\tilde{\in}\tilde{X}$ and $n \geq 1$, by iteration method, we have a sequence $\lbrace \tilde{x}_{n} \rbrace$ of soft elements in $\tilde{X}$ by letting $\tilde{x}_{1}=T\tilde{x}_{0}, \, \tilde{x}_{2}=T\tilde{x}_{1}=T^{2}\tilde{x}_{0}, \, \dots \, ,\tilde{x}_{n+1}=T\tilde{x}_{n}=T^{n+1}\tilde{x}_{0}, \, \dots$. Then, 
\begin{equation*}
d\left( \tilde{x}_{n+1},\tilde{x}_{n} \right) = d\left( T\tilde{x}_{n},T\tilde{x}_{n-1} \right) \tilde{\preceq} \tilde{t} d\left( \tilde{x}_{n},\tilde{x}_{n-1} \right)\tilde{\preceq} \tilde{t}^{2} d\left( \tilde{x}_{n-1},\tilde{x}_{n-2} \right)\tilde{\preceq} \dots \tilde{\preceq} \tilde{t}^{n} d\left( \tilde{x}_{1},\tilde{x}_{0} \right).
\end{equation*}
So, for $n>m$,
\begin{align*}
d\left( \tilde{x}_{n},\tilde{x}_{m} \right) &\tilde{\preceq} d\left( \tilde{x}_{n},\tilde{x}_{n-1} \right) + d\left( \tilde{x}_{n-1},\tilde{x}_{n-2} \right) + \cdots + d\left( \tilde{x}_{m+1},\tilde{x}_{m} \right) \\
&\tilde{\preceq} \left( \tilde{t}^{n-1}+\tilde{t}^{n-2}+ \cdots + \tilde{t}^{m} \right) d\left( \tilde{x}_{1},\tilde{x}_{0} \right) \\
&\tilde{\preceq} \frac{\tilde{t}^{m}}{\bar{1}-\tilde{t}}d\left( \tilde{x}_{1},\tilde{x}_{0} \right).
\end{align*}
Let $\tilde{c}\tilde{\in}\tilde{E}$ with $\Theta \tilde{\ll} \tilde{c}$ be given and choose $\tilde{\delta} \tilde{>} \bar{0}$ such that $\tilde{c}+N_{\tilde{\delta}}(\Theta) \tilde{\subset} (P,A)$, where $N_{\tilde{\delta}}(\Theta)=SS(\lbrace \tilde{y}\tilde{\in}\tilde{E}: \Vert \tilde{y} \Vert \tilde{<} \tilde{\delta} \rbrace)$. Also, choose a natural number $N_{1}$ such that $\frac{\tilde{t}^{m}}{\bar{1}-\tilde{t}}d\left( \tilde{x}_{1},\tilde{x}_{0} \right) \tilde{\in} N_{\tilde{\delta}}(\Theta)$, $\forall m\geq N_{1}$. Then, $\frac{\tilde{t}^{m}}{\bar{1}-\tilde{t}}d\left( \tilde{x}_{1},\tilde{x}_{0} \right) \tilde{\ll} \tilde{c}$, $\forall m>N_{1}$. Thus,
\begin{equation*}
d\left( \tilde{x}_{n},\tilde{x}_{m} \right) \tilde{\preceq} \frac{\tilde{t}^{m}}{\bar{1}-\tilde{t}}d\left( \tilde{x}_{1},\tilde{x}_{0} \right) \tilde{\ll} \tilde{c}, \quad  \forall n>m.
\end{equation*}
Therefore, $\lbrace \tilde{x}_{n} \rbrace$ is a Cauchy sequence in $\tilde{X}$. Since, $(\tilde{X},d,A)$ is a complete soft cone metric space, then exists $\tilde{x}^{*}\tilde{\in}\tilde{X}$ such that $\tilde{x}_{n} \to \tilde{x}^{*}$ as $n \to \infty$. Choose a natural number $N_{2}$ such that $d\left( \tilde{x}_{n},\tilde{x}^{*} \right)\tilde{\ll}\tilde{\frac{c}{2}}$, $\forall n \geq N_{2}$. Hence, 
\begin{align*}
d\left( T\tilde{x}^{*},\tilde{x}^{*} \right) &\tilde{\preceq} d\left( T\tilde{x}_{n},T\tilde{x}^{*} \right) + d\left( T\tilde{x}_{n},\tilde{x}^{*} \right) \\
&\tilde{\preceq} \tilde{t} d\left( \tilde{x}_{n},\tilde{x}^{*} \right) + d\left( \tilde{x}_{n+1},\tilde{x}^{*} \right) \\
&\tilde{\preceq} d\left( \tilde{x}_{n},\tilde{x}^{*} \right) + d\left( \tilde{x}_{n+1},\tilde{x}^{*} \right) \\
&\tilde{\ll} \tilde{\frac{c}{2}} + \tilde{\frac{c}{2}}=\tilde{c}, \quad \forall n \geq N_{2}.
\end{align*}
Thus, $d\left( T\tilde{x}^{*},\tilde{x}^{*} \right)\tilde{\ll}\tilde{\frac{c}{m}}$, $\forall m \geq 1$. So, $\tilde{\frac{c}{m}}
-d\left( T\tilde{x}^{*},\tilde{x}^{*} \right) \tilde{\in} (P,A)$, $\forall m \geq 1$. Since, $\tilde{\frac{c}{m}} \to \Theta$ as $m \to \infty$ and $(P,A)$ is closed, $-d\left( T\tilde{x}^{*},\tilde{x}^{*} \right)\tilde{\in}(P,A)$. But, $d\left( T\tilde{x}^{*},\tilde{x}^{*} \right)\tilde{\in}(P,A)$. Therefore, $d\left( T\tilde{x}^{*},\tilde{x}^{*} \right)=\Theta$ and so, $T\tilde{x}^{*}=\tilde{x}^{*}$.
\end{proof}

\begin{corollary}
Let $(\tilde{X},d,A)$ be a complete soft cone metric space. For $\Theta \tilde{\ll} \tilde{c}$ and $\tilde{x}_{0}\tilde{\in}\tilde{X}$, set $B(\tilde{x}_{0},\tilde{c})=\lbrace \tilde{x}\tilde{\in}\tilde{X}: d(\tilde{x}_{0},\tilde{x})\tilde{\ll}\tilde{c} \rbrace$ and $(P,A)=SS(B(\tilde{x}_{0},\tilde{c}))$. Suppose that the mapping $T:(\tilde{X},d,A) \rightarrow (\tilde{X},d,A)$ satisfies the contractive condition
\begin{equation*}
d\left( T\tilde{x},T\tilde{y} \right) = \tilde{t} d\left( \tilde{x},\tilde{y} \right), \quad \forall \tilde{x},\tilde{y}\tilde{\in}(P,A),
\end{equation*}
where $\bar{0}\tilde{\leq}\tilde{t}\tilde{<}\bar{1}$ is a soft constant and $d\left( T\tilde{x}_{0},\tilde{x}_{0} \right)\tilde{\preceq} (\bar{1} - \tilde{t})\tilde{c}$. Then, $T$ has a unique fixed soft element in $(P,A)$.
\end{corollary}

\begin{proof}
We prove that $(P,A)$ is complete and $T\tilde{x}\tilde{\in}(P,A)$, $\forall \tilde{x}\tilde{\in}(P,A)$. Suppose the sequence $\lbrace \tilde{x}_{n} \rbrace$ of soft elements of $(P,A)$ is a Cauchy sequence in $(P,A)$. Then, $\lbrace \tilde{x}_{n} \rbrace$ is also a Cauchy sequence in $\tilde{X}$. By the completeness of $\tilde{X}$, there is $\tilde{x}\tilde{\in}\tilde{X}$ such that $\tilde{x}_{n} \to \tilde{x}$ as $n \to \infty$. We have 
\begin{equation*}
d\left( \tilde{x}_{0},\tilde{x} \right) \tilde{\preceq} d\left( \tilde{x}_{0},\tilde{x}_{n} \right) + d\left( \tilde{x}_{n},\tilde{x} \right) \tilde{\preceq} d\left( \tilde{x}_{n},\tilde{x} \right) + \tilde{c}.
\end{equation*}
Since, $\tilde{x}_{n} \to \tilde{x}$, $d\left( \tilde{x}_{n},\tilde{x} \right) \to \Theta$. Hence, $d\left( \tilde{x}_{0},\tilde{x} \right)\tilde{\preceq}\tilde{c}$ and $\tilde{x}\tilde{\in}(P,A)$. Therefore, $(P,A)$ is complete. For every $\tilde{x}\tilde{\in} (P,A)$,
\begin{equation*}
d\left( \tilde{x}_{0},T\tilde{x} \right) \tilde{\preceq} d\left( T\tilde{x}_{0},\tilde{x}_{0} \right) + d\left( T\tilde{x}_{0},T\tilde{x} \right) \tilde{\preceq} (\bar{1}-\tilde{t})\tilde{c}+\tilde{t}\tilde{c}=\tilde{c}.
\end{equation*}
Hence, $T\tilde{x}\tilde{\in}(P,A)$.
\end{proof}

\begin{corollary}
Let $(\tilde{X},d,A)$ be a complete soft cone metric space. Suppose a mapping $T:(\tilde{X},d,A) \rightarrow (\tilde{X},d,A)$ satisfies the contractive condition for some positive integer $n$
\begin{equation*}
d\left( T^{n}\tilde{x},T^{n}\tilde{y} \right) \tilde{\preceq} \tilde{t} d\left( \tilde{x},\tilde{y} \right), \quad \forall \tilde{x},\tilde{y}\tilde{\in}\tilde{X},
\end{equation*}
where $\bar{0}\tilde{\leq}\tilde{t}\tilde{\leq}\bar{1}$ is a soft constant. Then, $T$ has a unique fixed soft element in $\tilde{X}$.
\end{corollary}

\begin{proof}
From \Cref{Teo4.4}, $T^{n}$ has a unique fixed soft element in $\tilde{X}$. But, $T^{n}(T\tilde{x}^{*})=T(T^{n}\tilde{x}^{*})=T\tilde{x}^{*}$, so $T\tilde{x}^{*}$ is also a fixed soft element of $T^{n}$. Hence, $T\tilde{x}^{*}=\tilde{x}$, $\tilde{x}$ is a fixed soft element of $T$. Since the fixed soft element of $T$ is also fixed soft element of $T^{n}$, the fixed soft element of $T$ is unique. 
\end{proof}

\begin{theorem}
Let $(\tilde{X},d,A)$ be a complete soft cone metric space and $T:(\tilde{X},d,A) \rightarrow (\tilde{X},d,A)$ satisfies the contractive condition 
\begin{equation*}
d\left( T\tilde{x},T\tilde{y} \right) \tilde{\preceq} \tilde{t} \left( d\left( T\tilde{x},\tilde{x} \right) + d\left( T\tilde{y},\tilde{y} \right) \right), \quad \forall \tilde{x},\tilde{y}\tilde{\in}\tilde{X},
\end{equation*}
where $\bar{0}\tilde{\leq}\tilde{t}\tilde{<}\bar{\frac{1}{2}}$ is a soft constant. Then, $T$ has a unique fixed soft element in $\tilde{X}$. For each $\tilde{x}\tilde{\in}\tilde{X}$, the iterative sequence $\lbrace T^{n}\tilde{x} \rbrace$ converges to the fixed soft element.
\end{theorem}

\begin{proof}
For each $\tilde{x}_{0}\tilde{\in}\tilde{X}$, set $\tilde{x}_{1}=T\tilde{x}_{0}, \, \tilde{x}_{2}=T\tilde{x}_{1}=T^{2}\tilde{x}_{0}, \, \dots \, ,\tilde{x}_{n+1}=T\tilde{x}_{n}=T^{n+1}\tilde{x}_{0}, \, \dots$. Then,
\begin{align*}
d\left( \tilde{x}_{n+1},\tilde{x}_{n} \right) =& d\left( T\tilde{x}_{n},T\tilde{x}_{n-1} \right) \\
\tilde{\preceq}& \tilde{t} \left( d\left( T\tilde{x}_{n},\tilde{x}_{n} \right) + d\left( T\tilde{x}_{n-1},\tilde{x}_{n-1} \right) \right) \\
=& \tilde{t} \left( d\left( \tilde{x}_{n+1},\tilde{x}_{n} \right) + d\left( \tilde{x}_{n},\tilde{x}_{n-1} \right) \right).
\end{align*}
So,
\begin{equation*}
d\left( \tilde{x}_{n+1},\tilde{x}_{n} \right) \tilde{\preceq} \frac{\tilde{t}}{\bar{1}-\tilde{t}} d\left( \tilde{x}_{n},\tilde{x}_{n-1} \right) = \tilde{s}d\left( \tilde{x}_{n},\tilde{x}_{n-1} \right),
\end{equation*}
where $\tilde{s}=\frac{\tilde{t}}{(\bar{1}-\tilde{t})}$. For $n>m$,
\begin{align*}
d\left( \tilde{x}_{n},\tilde{x}_{m} \right) &\tilde{\preceq} d\left( \tilde{x}_{n},\tilde{x}_{n-1} \right) + d\left( \tilde{x}_{n-1},\tilde{x}_{n-2} \right) + \cdots + d\left( \tilde{x}_{m+1},\tilde{x}_{m} \right) \\
&\tilde{\preceq} \left( \tilde{s}^{n-1}+\tilde{s}^{n-2}+ \cdots + \tilde{s}^{m} \right) d\left( \tilde{x}_{1},\tilde{x}_{0} \right) \\
&\tilde{\preceq} \frac{\tilde{s}^{m}}{\bar{1}-\tilde{s}} d\left( \tilde{x}_{1},\tilde{x}_{0} \right).
\end{align*}
Let $\Theta\tilde{\ll}\tilde{c}$ be given. Choose a natural number $N_{1}$ such that $\frac{\tilde{s}^{m}}{\bar{1}-\tilde{s}}d\left( \tilde{x}_{1},\tilde{x}_{0} \right) \tilde{\ll} \tilde{c}$, $\forall m\geq N_{1}$. Thus, $d\left( \tilde{x}_{n},\tilde{x}_{m} \right)\tilde{\ll}\tilde{c}$, for $n>m$. Therefore, $\lbrace \tilde{x}_{n} \rbrace$ is a Cauchy sequence in $\tilde{X}$. Since, $(\tilde{X},d,A)$ is a complete soft cone metric space, there is $\tilde{x}^{*}\tilde{\in}\tilde{X}$ such that $\tilde{x}_{n} \to \tilde{x}^{*}$ as $n \to \infty$. Choose a natural number $N_{2}$ such that $d\left( \tilde{x}_{n+1},\tilde{x}_{n} \right)\tilde{\ll}(\bar{1}-\tilde{t})\tilde{\frac{c}{2}}$ and $d\left( \tilde{x}_{n+1},\tilde{x}^{*} \right)\tilde{\ll}(\bar{1}-\tilde{t})\tilde{\frac{c}{2}}$, $\forall n \geq N_{2}$. Hence, for $n \geq N_{2}$, we have
\begin{align*}
d\left( T\tilde{x}^{*},\tilde{x}^{*} \right) &\tilde{\preceq} d\left( T\tilde{x}_{n},T\tilde{x}^{*} \right) + d\left( T\tilde{x}_{n},\tilde{x}^{*} \right) \\
&\tilde{\preceq} \tilde{t}\left( d\left( T\tilde{x}_{n},\tilde{x}_{n} \right) + d\left( T\tilde{x}^{*},\tilde{x}^{*} \right) \right) + d\left( \tilde{x}_{n+1},\tilde{x}^{*} \right).
\end{align*}
Hence, 
\begin{equation*}
d\left( T\tilde{x}_{n},\tilde{x}^{*} \right)\tilde{\preceq} \frac{\bar{1}}{\bar{1}-\tilde{t}}\left( \tilde{t}d\left( \tilde{x}_{n+1},\tilde{x}_{n} \right) + d\left( \tilde{x}_{n+1},\tilde{x}^{*} \right) \right) \tilde{\ll} \tilde{\frac{c}{2}}+\tilde{\frac{c}{2}}=\tilde{c}.
\end{equation*}
Thus, $d\left( T\tilde{x}^{*},\tilde{x}^{*} \right)\tilde{\ll}\tilde{\frac{c}{m}}$, $\forall m \geq 1$. So, $\tilde{\frac{c}{m}}
-d\left( T\tilde{x}^{*},\tilde{x}^{*} \right) \tilde{\in} (P,A)$, $\forall m \geq 1$. Since, $\tilde{\frac{c}{m}} \to \Theta$ as $m \to \infty$ and $(P,A)$ is closed, $-d\left( T\tilde{x}^{*},\tilde{x}^{*} \right)\tilde{\in}(P,A)$. But, $d\left( T\tilde{x}^{*},\tilde{x}^{*} \right)\tilde{\in}(P,A)$. Therefore, $d\left( T\tilde{x}^{*},\tilde{x}^{*} \right)=\Theta$ and so, $T\tilde{x}^{*}=\tilde{x}^{*}$.

If $\tilde{y}^{*}$ is another fixed soft element of $T$, then
\begin{equation*}
d\left( \tilde{x}^{*},\tilde{y}^{*} \right)=d\left( T\tilde{x}^{*},T\tilde{y}^{*} \right)\tilde{\preceq} \tilde{t}\left( d\left( T\tilde{x}^{*},\tilde{x}^{*} \right) + d\left( T\tilde{y}^{*},\tilde{y}^{*} \right) \right)=\Theta.
\end{equation*}
Hence, $\tilde{x}^{*}=\tilde{y}^{*}$. Therefore, the fixed soft element of $T$ is unique.
\end{proof}

\begin{theorem}
Let $(\tilde{X},d,A)$ be a complete soft cone metric space and $T:(\tilde{X},d,A) \rightarrow (\tilde{X},d,A)$ satisfies the contractive condition 
\begin{equation*}
d\left( T\tilde{x},T\tilde{y} \right) \tilde{\preceq} \tilde{t} \left( d\left( T\tilde{x},\tilde{y} \right) + d\left( T\tilde{y},\tilde{x} \right) \right), \quad \forall \tilde{x},\tilde{y}\tilde{\in}\tilde{X},
\end{equation*}
where $\bar{0}\tilde{\leq}\tilde{t}\tilde{<}\bar{\frac{1}{2}}$ is a soft constant. Then, $T$ has a unique fixed soft element in $\tilde{X}$. For each $\tilde{x}\tilde{\in}\tilde{X}$, the iterative sequence $\lbrace T^{n}\tilde{x} \rbrace$ converges to the fixed soft element.
\end{theorem}

\begin{proof}
For each $\tilde{x}_{0}\tilde{\in}\tilde{X}$ and $n \geq 1$ , set $\tilde{x}_{1}=T\tilde{x}_{0}, \, \tilde{x}_{2}=T\tilde{x}_{1}=T^{2}\tilde{x}_{0}, \, \dots \, ,\tilde{x}_{n+1}=T\tilde{x}_{n}=T^{n+1}\tilde{x}_{0}, \, \dots$. Then,
\begin{align*}
d\left( \tilde{x}_{n+1},\tilde{x}_{n} \right) =& d\left( T\tilde{x}_{n},T\tilde{x}_{n-1} \right) \\
\tilde{\preceq}& \tilde{t} \left( d\left( T\tilde{x}_{n},\tilde{x}_{n-1} \right) + d\left( T\tilde{x}_{n-1},\tilde{x}_{n} \right) \right) \\
\tilde{\preceq}& \tilde{t} \left( d\left( \tilde{x}_{n+1},\tilde{x}_{n} \right) + d\left( \tilde{x}_{n},\tilde{x}_{n-1} \right) \right).
\end{align*}
So,
\begin{equation*}
d\left( \tilde{x}_{n+1},\tilde{x}_{n} \right) \tilde{\preceq} \frac{\tilde{t}}{\bar{1}-\tilde{t}} d\left( \tilde{x}_{n},\tilde{x}_{n-1} \right) = \tilde{s}d\left( \tilde{x}_{n},\tilde{x}_{n-1} \right),
\end{equation*}
where $\tilde{s}=\frac{\tilde{t}}{(\bar{1}-\tilde{t})}$. For $n>m$,
\begin{align*}
d\left( \tilde{x}_{n},\tilde{x}_{m} \right) &\tilde{\preceq} d\left( \tilde{x}_{n},\tilde{x}_{n-1} \right) + d\left( \tilde{x}_{n-1},\tilde{x}_{n-2} \right) + \cdots + d\left( \tilde{x}_{m+1},\tilde{x}_{m} \right) \\
&\tilde{\preceq} \left( \tilde{s}^{n-1}+\tilde{s}^{n-2}+ \cdots + \tilde{s}^{m} \right) d\left( \tilde{x}_{1},\tilde{x}_{0} \right) \\
&\tilde{\preceq} \frac{\tilde{s}^{m}}{\bar{1}-\tilde{s}} d\left( \tilde{x}_{1},\tilde{x}_{0} \right).
\end{align*}
Let $\Theta\tilde{\ll}\tilde{c}$ be given. Choose a natural number $N_{1}$ such that $\frac{\tilde{s}^{m}}{\bar{1}-\tilde{s}}d\left( \tilde{x}_{1},\tilde{x}_{0} \right) \tilde{\ll} \tilde{c}$, $\forall m\geq N_{1}$. Thus, $d\left( \tilde{x}_{n},\tilde{x}_{m} \right)\tilde{\ll}\tilde{c}$, for $n>m$. Therefore, $\lbrace \tilde{x}_{n} \rbrace$ is a Cauchy sequence in $\tilde{X}$. Since, $(\tilde{X},d,A)$ is a complete soft cone metric space, there is $\tilde{x}^{*}\tilde{\in}\tilde{X}$ such that $\tilde{x}_{n} \to \tilde{x}^{*}$ as $n \to \infty$. Choose a natural number $N_{2}$ such that $d\left( \tilde{x}_{n},\tilde{x}^{*} \right)\tilde{\ll}(\bar{1}-\tilde{t})\tilde{\frac{c}{3}}$. Hence, for $n \geq N_{2}$, we have
\begin{align*}
d\left( T\tilde{x}^{*},\tilde{x}^{*} \right) &\tilde{\preceq} d\left( T\tilde{x}_{n},T\tilde{x}^{*} \right) + d\left( T\tilde{x}_{n},\tilde{x}^{*} \right) \\
&\tilde{\preceq} \tilde{t}\left( d\left( T\tilde{x}^{*},\tilde{x}_{n} \right) + d\left( T\tilde{x}_{n},\tilde{x}^{*} \right) \right) + d\left( \tilde{x}_{n+1},\tilde{x}^{*} \right) \\
&\tilde{\preceq} \tilde{t}\left( d\left( T\tilde{x}^{*},\tilde{x}^{*} \right) + d\left( \tilde{x}_{n},\tilde{x}^{*} \right) + d\left( \tilde{x}_{n+1},\tilde{x}^{*} \right) \right) + d\left( \tilde{x}_{n+1},\tilde{x}^{*} \right).
\end{align*}
Thus, 
\begin{equation*}
d\left( T\tilde{x}^{*},\tilde{x}^{*} \right)\tilde{\preceq} \frac{\bar{1}}{\bar{1}-\tilde{t}}\left( \tilde{t}d\left( \tilde{x}_{n},\tilde{x}^{*} \right) + d\left( \tilde{x}_{n+1},\tilde{x}^{*} \right) \right) + d\left( \tilde{x}_{n+1},\tilde{x}^{*} \right) \tilde{\ll} \tilde{\frac{c}{3}}+\tilde{\frac{c}{3}}+\tilde{\frac{c}{3}}=\tilde{c}.
\end{equation*}
Thus, $d\left( T\tilde{x}^{*},\tilde{x}^{*} \right)\tilde{\ll}\tilde{\frac{c}{m}}$, $\forall m \geq 1$. So, $\tilde{\frac{c}{m}}
-d\left( T\tilde{x}^{*},\tilde{x}^{*} \right) \tilde{\in} (P,A)$. Since, $\tilde{\frac{c}{m}} \to \Theta$ as $m \to \infty$ and $(P,A)$ is closed, $-d\left( T\tilde{x}^{*},\tilde{x}^{*} \right)\tilde{\in}(P,A)$. But, $d\left( T\tilde{x}^{*},\tilde{x}^{*} \right)\tilde{\in}(P,A)$. Therefore, $d\left( T\tilde{x}^{*},\tilde{x}^{*} \right)=\Theta$ and so, $T\tilde{x}^{*}=\tilde{x}^{*}$.

If $\tilde{y}^{*}$ is another fixed soft element of $T$, then
\begin{equation*}
d\left( \tilde{x}^{*},\tilde{y}^{*} \right)=d\left( T\tilde{x}^{*},T\tilde{y}^{*} \right)\tilde{\preceq} \tilde{t}\left( d\left( T\tilde{x}^{*},\tilde{y}^{*} \right) + d\left( T\tilde{y}^{*},\tilde{x}^{*} \right) \right)=2\tilde{t} d\left( \tilde{x}^{*},\tilde{y}^{*} \right).
\end{equation*}
Hence, $d\left( \tilde{x}^{*},\tilde{y}^{*} \right)=\Theta$ and so $\tilde{x}^{*}=\tilde{y}^{*}$. Therefore, the fixed soft element of $T$ is unique.
\end{proof}

\begin{theorem}
Let $(\tilde{X},d,A)$ be a complete soft cone metric space and $T:(\tilde{X},d,A) \rightarrow (\tilde{X},d,A)$ satisfies the contractive condition 
\begin{equation*}
d\left( T\tilde{x},T\tilde{y} \right) \tilde{\preceq} \tilde{t} d\left( \tilde{x},\tilde{y} \right) + \tilde{r} d\left( \tilde{y},T\tilde{x} \right), \quad \forall \tilde{x},\tilde{y}\tilde{\in}\tilde{X},
\end{equation*}
where $\bar{0}\tilde{\leq}\tilde{t},\tilde{r}\tilde{<}\bar{1}$ are soft constants. Then, $T$ has a fixed soft element in $\tilde{X}$. Also, the soft fixed element of $T$ is unique whenever $\tilde{t}+\tilde{r}\tilde{<}\bar{1}$.
\end{theorem}

\begin{proof}
For each $\tilde{x}_{0}\tilde{\in}\tilde{X}$ and $n \geq 1$ , set $\tilde{x}_{1}=T\tilde{x}_{0}, \, \tilde{x}_{2}=T\tilde{x}_{1}=T^{2}\tilde{x}_{0}, \, \dots \, ,\tilde{x}_{n+1}=T\tilde{x}_{n}=T^{n+1}\tilde{x}_{0}, \, \dots$. Then,
\begin{align*}
d\left( \tilde{x}_{n+1},\tilde{x}_{n} \right) =& d\left( T\tilde{x}_{n},T\tilde{x}_{n-1} \right) \\
\tilde{\preceq}& \tilde{t} \left( d\left( \tilde{x}_{n},\tilde{x}_{n-1} \right) + d\left( T\tilde{x}_{n-1},\tilde{x}_{n} \right) \right) \\
=&\tilde{t} d\left( \tilde{x}_{n},\tilde{x}_{n-1} \right) \\
\tilde{\preceq}& \tilde{t}^{n}  d\left( \tilde{x}_{1},\tilde{x}_{0} \right).
\end{align*}
Thus, for $n>m$, we have
\begin{align*}
d\left( \tilde{x}_{n},\tilde{x}_{m} \right) &\tilde{\preceq} d\left( \tilde{x}_{n},\tilde{x}_{n-1} \right) + d\left( \tilde{x}_{n-1},\tilde{x}_{n-2} \right) + \cdots + d\left( \tilde{x}_{m+1},\tilde{x}_{m} \right) \\
&\tilde{\preceq} \left( \tilde{t}^{n-1}+\tilde{t}^{n-2}+ \cdots + \tilde{t}^{m} \right) d\left( \tilde{x}_{1},\tilde{x}_{0} \right) \\
&\tilde{\preceq} \frac{\tilde{t}^{m}}{\bar{1}-\tilde{t}} d\left( \tilde{x}_{1},\tilde{x}_{0} \right).
\end{align*}
Let $\Theta\tilde{\ll}\tilde{c}$ be given. Choose a natural number $N_{1}$ such that $\frac{\tilde{t}^{m}}{\bar{1}-\tilde{t}}d\left( \tilde{x}_{1},\tilde{x}_{0} \right) \tilde{\ll} \tilde{c}$, $\forall m\geq N_{1}$. Thus, $\left( \tilde{x}_{n},\tilde{x}_{m} \right)\tilde{\ll}\tilde{c}$, for $n>m$. Therefore, $\lbrace \tilde{x}_{n} \rbrace$ is a Cauchy sequence in $\tilde{X}$. Since, $(\tilde{X},d,A)$ is a complete soft cone metric space, there is $\tilde{x}^{*}\tilde{\in}\tilde{X}$ such that $\tilde{x}_{n} \to \tilde{x}^{*}$ as $n \to \infty$. Choose a natural number $N_{2}$ such that $d\left( \tilde{x}_{n},\tilde{x}^{*} \right)\tilde{\ll}(\bar{1}-\tilde{t})\tilde{\frac{c}{3}}$. Hence, for $n \geq N_{2}$, we have
\begin{align*}
d\left( T\tilde{x}^{*},\tilde{x}^{*} \right) \tilde{\preceq}& d\left( \tilde{x}_{n},T\tilde{x}^{*} \right) + d\left( \tilde{x}_{n},\tilde{x}^{*} \right) \\
=& d\left( T\tilde{x}_{n-1},T\tilde{x}^{*} \right)+d\left( \tilde{x}_{n},\tilde{x}^{*} \right) \\
\tilde{\preceq}& \tilde{t} d\left( \tilde{x}_{n-1},\tilde{x}^{*} \right) + \tilde{r} d\left( T\tilde{x}_{n-1},\tilde{x}^{*} \right) + d\left( \tilde{x}_{n},\tilde{x}^{*} \right) \\
\tilde{\preceq}& d\left( \tilde{x}_{n-1},\tilde{x}^{*} \right) + d\left( \tilde{x}_{n},\tilde{x}^{*} \right) + d\left( \tilde{x}_{n},\tilde{x}^{*} \right) \\
\tilde{\ll}& \tilde{\frac{c}{3}}+\tilde{\frac{c}{3}}+\tilde{\frac{c}{3}}=\tilde{c}.
\end{align*}
Thus, $d\left( T\tilde{x}^{*},\tilde{x}^{*} \right)\tilde{\ll}\tilde{\frac{c}{m}}$, $\forall m \geq 1$. So, $\tilde{\frac{c}{m}}
-d\left( T\tilde{x}^{*},\tilde{x}^{*} \right) \tilde{\in} (P,A)$. Since, $\tilde{\frac{c}{m}} \to \Theta$ as $m \to \infty$ and $(P,A)$ is closed, $-d\left( T\tilde{x}^{*},\tilde{x}^{*} \right)\tilde{\in}(P,A)$. But, $d\left( T\tilde{x}^{*},\tilde{x}^{*} \right)\tilde{\in}(P,A)$. Therefore, $d\left( T\tilde{x}^{*},\tilde{x}^{*} \right)=\Theta$ and so, $T\tilde{x}^{*}=\tilde{x}^{*}$.

If $\tilde{y}^{*}$ is another fixed soft element of $T$ and $\tilde{t}+\tilde{r}\tilde{<}\bar{1}$, then
\begin{equation*}
d\left( \tilde{x}^{*},\tilde{y}^{*} \right)=d\left( T\tilde{x}^{*},T\tilde{y}^{*} \right)\tilde{\preceq} \tilde{t} d\left( \tilde{x}^{*},\tilde{y}^{*} \right) + \tilde{r} d\left( T\tilde{x}^{*},\tilde{y}^{*} \right) =(\tilde{t}+\tilde{r}) d\left( \tilde{x}^{*},\tilde{y}^{*} \right).
\end{equation*}
Hence, $d\left( \tilde{x}^{*},\tilde{y}^{*} \right)=\Theta$ and so $\tilde{x}^{*}=\tilde{y}^{*}$. Therefore, the fixed soft element of $T$ is unique whenever $\tilde{t}+\tilde{r}\tilde{<}\bar{1}$.
\end{proof}

\section*{Conclusion}
In this paper we introduced the concept of soft cone metric spaces via soft element and we worked the convergence of sequences and Cauchy sequences in such spaces. We also discussed some fixed point theorems of contractive mapping on soft cone metric spaces. There is ample scope for further research on soft cone metric spaces. This paper is a basis to works on the above mentioned ideas.

%\bibliographystyle{elsarticle-num}
%\bibliography{mybibfile}

\end{document}